\documentclass[reqno, 12pt]{amsart}
\pdfoutput=1

\def\ssign{\textsection\nobreak\hspace{1pt plus 0.3pt}}
\makeatletter
\let\origsection=\section 
\def\mysection{\@mystartsection{section}{1}\z@{.7\linespacing\@plus\linespacing}{.5\linespacing}{\normalfont\scshape\centering\ssign}}
\def\section{\@ifstar{\origsection*}{\mysection}}
\def\appendix{\par\c@section\z@ \c@subsection\z@
   \let\sectionname\appendixname
   \let\section=\origsection
   \def\thesection{\@Alph\c@section}} 
\def\@mystartsection#1#2#3#4#5#6{\if@noskipsec \leavevmode \fi
 \par \@tempskipa #4\relax
 \@afterindenttrue
 \ifdim \@tempskipa <\z@ \@tempskipa -\@tempskipa \@afterindentfalse\fi
 \if@nobreak \everypar{}\else
     \addpenalty\@secpenalty\addvspace\@tempskipa\fi
 \@dblarg{\@mysect{#1}{#2}{#3}{#4}{#5}{#6}}}
\def\@mysect#1#2#3#4#5#6[#7]#8{\edef\@toclevel{\ifnum#2=\@m 0\else\number#2\fi}\ifnum #2>\c@secnumdepth \let\@secnumber\@empty
  \else \@xp\let\@xp\@secnumber\csname the#1\endcsname\fi
  \@tempskipa #5\relax
  \ifnum #2>\c@secnumdepth
    \let\@svsec\@empty
  \else
    \refstepcounter{#1}\edef\@secnumpunct{\ifdim\@tempskipa>\z@ \@ifnotempty{#8}{\@nx\enspace}\else
        \@ifempty{#8}{.}{\@nx\enspace}\fi
    }\@ifempty{#8}{\ifnum #2=\tw@ \def\@secnumfont{\bfseries}\fi}{}\protected@edef\@svsec{\ifnum#2<\@m
        \@ifundefined{#1name}{}{\ignorespaces\csname #1name\endcsname\space
        }\fi
      \@seccntformat{#1}}\fi
  \ifdim \@tempskipa>\z@ \begingroup #6\relax
    \@hangfrom{\hskip #3\relax\@svsec}{\interlinepenalty\@M #8\par}\endgroup
    \ifnum#2>\@m \else \@tocwrite{#1}{#8}\fi
  \else
  \def\@svsechd{#6\hskip #3\@svsec
    \@ifnotempty{#8}{\ignorespaces#8\unskip
       \@addpunct.}\ifnum#2>\@m \else \@tocwrite{#1}{#8}\fi
  }\fi
  \global\@nobreaktrue
  \@xsect{#5}}
\makeatother

\usepackage{amsmath,amssymb,amsthm}
\usepackage{mathrsfs}
\usepackage{mathabx}\changenotsign
\usepackage{dsfont}
\usepackage{bbm}
\usepackage{todonotes}

\usepackage{xcolor}
\usepackage[backref=section]{hyperref}
\usepackage[ocgcolorlinks]{ocgx2}
\hypersetup{
	colorlinks=true,
	linkcolor={red!60!black},
	citecolor={green!60!black},
	urlcolor={blue!60!black},
}

\usepackage[open,openlevel=2,atend]{bookmark}

\usepackage[abbrev,msc-links,backrefs]{amsrefs}
\usepackage{doi}

\renewcommand{\PrintDOI}[1]{\doi{#1}}

\usepackage[T1]{fontenc}
\usepackage{lmodern}
\usepackage[babel]{microtype}
\usepackage[english]{babel}

\usepackage{geometry}
\numberwithin{equation}{section}
\numberwithin{figure}{section}

\usepackage{enumitem}

\let\polishlcross=\l
\def\l{\ifmmode\ell\else\polishlcross\fi}

\let\emptyset=\varnothing
\let\setminus=\smallsetminus

\makeatletter
\def\moverlay{\mathpalette\mov@rlay}
\def\mov@rlay#1#2{\leavevmode\vtop{   \baselineskip\z@skip \lineskiplimit-\maxdimen
		\ialign{\hfil$\m@th#1##$\hfil\cr#2\crcr}}}
\newcommand{\charfusion}[3][\mathord]{
	#1{\ifx#1\mathop\vphantom{#2}\fi
		\mathpalette\mov@rlay{#2\cr#3}
	}
	\ifx#1\mathop\expandafter\displaylimits\fi}
\makeatother

\newcommand{\dcup}{\charfusion[\mathbin]{\cup}{\cdot}}

\DeclareFontFamily{U}  {MnSymbolC}{}
\DeclareSymbolFont{MnSyC}         {U}  {MnSymbolC}{m}{n}
\DeclareFontShape{U}{MnSymbolC}{m}{n}{
	<-6>  MnSymbolC5
	<6-7>  MnSymbolC6
	<7-8>  MnSymbolC7
	<8-9>  MnSymbolC8
	<9-10> MnSymbolC9
	<10-12> MnSymbolC10
	<12->   MnSymbolC12}{}
\DeclareMathSymbol{\powerset}{\mathord}{MnSyC}{180}

\usepackage{tikz}
\usetikzlibrary{calc,decorations.pathmorphing}
\usetikzlibrary{arrows,decorations.pathreplacing}
\pgfdeclarelayer{background}
\pgfdeclarelayer{foreground}
\pgfdeclarelayer{front}
\pgfsetlayers{background,main,foreground,front}

\usepackage{multicol}
\usepackage{subcaption}

\newcommand{\pedge}[9]{
	
	\ifx\relax#6\relax
	\def\qoffs{0pt}
	\else
	\def\qoffs{#6}
	\fi
	
	\def\phedge{
		($#1+#5!\qoffs!-90:#2-#5$) -- 
		($#2+#1!\qoffs!-90:#3-#1$) -- 
		($#3+#2!\qoffs!-90:#4-#2$) -- 
		($#4+#3!\qoffs!-90:#5-#3$) -- 
		($#5+#4!\qoffs!-90:#1-#4$) -- cycle}

	\coordinate (12) at ($#1!\qoffs!90:#2$);
	\coordinate (15) at ($#1!\qoffs!-90:#5$);
	\coordinate (23) at ($#2!\qoffs!90:#3$);
	\coordinate (21) at ($#2!\qoffs!-90:#1$);
	\coordinate (34) at ($#3!\qoffs!90:#4$);
	\coordinate (32) at ($#3!\qoffs!-90:#2$);
	\coordinate (45) at ($#4!\qoffs!90:#5$);
	\coordinate (43) at ($#4!\qoffs!-90:#3$);
	\coordinate (51) at ($#5!\qoffs!90:#1$);
	\coordinate (54) at ($#5!\qoffs!-90:#4$);

	\def\nphedge{
		(15) let \p1=($(15)-#1$), \p2=($(12)-#1$) in 
		arc[start angle={atan2(\y1,\x1)}, delta angle={atan2(\y2,\x2)-atan2(\y1,\x1)-360*(atan2(\y2,\x2)-atan2(\y1,\x1)>0)}, x radius=\qoffs, y radius=\qoffs] --
		(21) let \p1=($(21)-#2$), \p2=($(23)-#2$) in 
		arc[start angle={atan2(\y1,\x1)}, delta angle={atan2(\y2,\x2)-atan2(\y1,\x1)-360*(atan2(\y2,\x2)-atan2(\y1,\x1)>0)}, x radius=\qoffs, y radius=\qoffs] --
		(32) let \p1=($(32)-#3$), \p2=($(34)-#3$) in 
		arc[start angle={atan2(\y1,\x1)}, delta angle={atan2(\y2,\x2)-atan2(\y1,\x1)-360*(atan2(\y2,\x2)-atan2(\y1,\x1)>0)}, x radius=\qoffs, y radius=\qoffs] --
		(43) let \p1=($(43)-#4$), \p2=($(45)-#4$) in 
		arc[start angle={atan2(\y1,\x1)}, delta angle={atan2(\y2,\x2)-atan2(\y1,\x1)-360*(atan2(\y2,\x2)-atan2(\y1,\x1)>0)}, x radius=\qoffs, y radius=\qoffs] --
		(54) let \p1=($(54)-#5$), \p2=($(51)-#5$) in 
		arc[start angle={atan2(\y1,\x1)}, delta angle={atan2(\y2,\x2)-atan2(\y1,\x1)-360*(atan2(\y2,\x2)-atan2(\y1,\x1)>0)}, x radius=\qoffs, y radius=\qoffs] --
		cycle}
	
	\ifx\relax#7\relax
	\def\plwidth{1pt}
	\else
	\def\plwidth{#7}
	\fi
	
	\ifx\relax#9\relax
	\fill \nphedge;
	\else
	\fill[#9]\nphedge;
	\fi
	
	\ifx\relax#8\relax
	\draw[line width=\plwidth,rounded corners=\qoffs]\nphedge;
	\else
	\draw[line width=\plwidth,#8]\nphedge;
	\fi
}

\newcommand{\qedge}[7]{
	
	\ifx\relax#4\relax
	\def\qoffs{0pt}
	\else
	\def\qoffs{#4}
	\fi
	
	\def\qhedge{
		($#1+#3!\qoffs!-90:#2-#3$) --
		($#2+#1!\qoffs!-90:#3-#1$) --
		($#3+#2!\qoffs!-90:#1-#2$) -- cycle}

	\coordinate (12) at ($#1!\qoffs!90:#2$);
	\coordinate (13) at ($#1!\qoffs!-90:#3$);
	\coordinate (23) at ($#2!\qoffs!90:#3$);
	\coordinate (21) at ($#2!\qoffs!-90:#1$);
	\coordinate (31) at ($#3!\qoffs!90:#1$);
	\coordinate (32) at ($#3!\qoffs!-90:#2$);
	
	\def\nqhedge{
		(13) let \p1=($(13)-#1$), \p2=($(12)-#1$) in
		arc[start angle={atan2(\y1,\x1)}, delta angle={atan2(\y2,\x2)-atan2(\y1,\x1)-360*(atan2(\y2,\x2)-atan2(\y1,\x1)>0)}, x radius=\qoffs, y radius=\qoffs] --
		(21) let \p1=($(21)-#2$), \p2=($(23)-#2$) in
		arc[start angle={atan2(\y1,\x1)}, delta angle={atan2(\y2,\x2)-atan2(\y1,\x1)-360*(atan2(\y2,\x2)-atan2(\y1,\x1)>0)}, x radius=\qoffs, y radius=\qoffs] --
		(32) let \p1=($(32)-#3$), \p2=($(31)-#3$) in
		arc[start angle={atan2(\y1,\x1)}, delta angle={atan2(\y2,\x2)-atan2(\y1,\x1)-360*(atan2(\y2,\x2)-atan2(\y1,\x1)>0)}, x radius=\qoffs, y radius=\qoffs] --
		cycle}
	
	\ifx\relax#5\relax
	\def\qlwidth{1pt}
	\else
	\def\qlwidth{#5}
	\fi
	
	\ifx\relax#7\relax
	\fill \nqhedge;
	\else
	\fill[#7]\nqhedge;
	\fi
	
	\ifx\relax#6\relax
	\draw[line width=\qlwidth,rounded corners=\qoffs]\nqhedge;
	\else
	\draw[line width=\qlwidth,#6]\nqhedge;
	\fi
}

\newcommand{\redge}[8]{
	
	\ifx\relax#5\relax
	\def\qoffs{0pt}
	\else
	\def\qoffs{#5}
	\fi
	
	\def\rhedge{
		($#1+#4!\qoffs!-90:#2-#4$) -- 
		($#2+#1!\qoffs!-90:#3-#1$) -- 
		($#3+#2!\qoffs!-90:#4-#2$) -- 
		($#4+#3!\qoffs!-90:#1-#3$) -- cycle}

	\coordinate (12) at ($#1!\qoffs!90:#2$);
	\coordinate (14) at ($#1!\qoffs!-90:#4$);
	\coordinate (23) at ($#2!\qoffs!90:#3$);
	\coordinate (21) at ($#2!\qoffs!-90:#1$);
	\coordinate (34) at ($#3!\qoffs!90:#4$);
	\coordinate (32) at ($#3!\qoffs!-90:#2$);
	\coordinate (41) at ($#4!\qoffs!90:#1$);
	\coordinate (43) at ($#4!\qoffs!-90:#3$);
	
	\def\nrhedge{
		(14) let \p1=($(14)-#1$), \p2=($(12)-#1$) in 
		arc[start angle={atan2(\y1,\x1)}, delta angle={atan2(\y2,\x2)-atan2(\y1,\x1)-360*(atan2(\y2,\x2)-atan2(\y1,\x1)>0)}, x radius=\qoffs, y radius=\qoffs] --
		(21) let \p1=($(21)-#2$), \p2=($(23)-#2$) in 
		arc[start angle={atan2(\y1,\x1)}, delta angle={atan2(\y2,\x2)-atan2(\y1,\x1)-360*(atan2(\y2,\x2)-atan2(\y1,\x1)>0)}, x radius=\qoffs, y radius=\qoffs] --
		(32) let \p1=($(32)-#3$), \p2=($(34)-#3$) in 
		arc[start angle={atan2(\y1,\x1)}, delta angle={atan2(\y2,\x2)-atan2(\y1,\x1)-360*(atan2(\y2,\x2)-atan2(\y1,\x1)>0)}, x radius=\qoffs, y radius=\qoffs] --
		(43) let \p1=($(43)-#4$), \p2=($(41)-#4$) in 
		arc[start angle={atan2(\y1,\x1)}, delta angle={atan2(\y2,\x2)-atan2(\y1,\x1)-360*(atan2(\y2,\x2)-atan2(\y1,\x1)>0)}, x radius=\qoffs, y radius=\qoffs] --
		cycle}
	
	\ifx\relax#6\relax
	\def\rlwidth{1pt}
	\else
	\def\rlwidth{#6}
	\fi
	
	\ifx\relax#8\relax
	\fill \nrhedge;
	\else
	\fill[#8]\nrhedge;
	\fi
	
	\ifx\relax#7\relax
	\draw[line width=\rlwidth,rounded corners=\qoffs]\nrhedge;
	\else
	\draw[line width=\rlwidth,#7]\nrhedge;
	\fi
}

\let\phi=\varphi
\let\epsilon=\varepsilon
\let\eps=\epsilon
\let\rho=\varrho
\let\theta=\vartheta

\def\NN{{\mathds N}}

\newcommand{\bbS}{\mathbb{S}}

\newcommand{\bx}{\mathbf{x}}

\newcommand{\cF}{\mathcal{F}}

\newcommand{\ccP}{\mathscr{P}}

\newcommand{\fin}{\text{fin}}

\newcommand{\ceil}[1]{\left\lceil #1 \right\rceil}

\newtheoremstyle{note}  {4pt}  {4pt}  {\sl}  {}  {\bfseries}  {.}  {.5em}          {}
\newtheoremstyle{introthms}  {3pt}  {3pt}  {\itshape}  {}  {\bfseries}  {.}  {.5em}          {\thmnote{#3}}
\newtheoremstyle{remark}  {2pt}  {2pt}  {\rm}  {}  {\bfseries}  {.}  {.3em}          {}

\theoremstyle{plain}
\newtheorem{theorem}{Theorem}[section]
\newtheorem{lemma}[theorem]{Lemma}

\newtheorem{claim}[theorem]{Claim}
\newtheorem{obs}[theorem]{Observation}

\theoremstyle{note}
\newtheorem{dfn}[theorem]{Definition}

\theoremstyle{remark}
\newtheorem{remark}[theorem]{Remark}

\usepackage{lineno}
\newcommand*\patchAmsMathEnvironmentForLineno[1]{
	\expandafter\let\csname old#1\expandafter\endcsname\csname #1\endcsname
	\expandafter\let\csname oldend#1\expandafter\endcsname\csname end#1\endcsname
	\renewenvironment{#1}
	{\linenomath\csname old#1\endcsname}
	{\csname oldend#1\endcsname\endlinenomath}}
\newcommand*\patchBothAmsMathEnvironmentsForLineno[1]{
	\patchAmsMathEnvironmentForLineno{#1}
	\patchAmsMathEnvironmentForLineno{#1*}}
\AtBeginDocument{
	\patchBothAmsMathEnvironmentsForLineno{equation}
	\patchBothAmsMathEnvironmentsForLineno{align}
	\patchBothAmsMathEnvironmentsForLineno{flalign}
	\patchBothAmsMathEnvironmentsForLineno{alignat}
	\patchBothAmsMathEnvironmentsForLineno{gather}
	\patchBothAmsMathEnvironmentsForLineno{multline}
}

\def\ex{\text{\rm ex}}

\usepackage{scalerel}

\newsavebox\vdegbox
\savebox\vdegbox{\tikz{
		\draw[black,fill=black] (90:1) circle (.35);
		\draw[black,line width=0.10cm] (210:1) circle (.30);
		\draw[black,line width=0.10cm] (330:1) circle (.30);
		\draw[opacity=0] (0:1.2) circle (0.1);
	}}

\newsavebox\vvbox
\savebox\vvbox{\tikz{
		\draw[black,line width=0.10cm] (90:1) circle (.30);
		\draw[black,fill=black] (210:1) circle (.35);
		\draw[black,fill=black] (330:1) circle (.35);
		\draw[opacity=0] (0:1.2) circle (0.1);
	}}

\newsavebox\pdegbox
\savebox\pdegbox{\tikz{
		\draw[black,line width=0.10cm] (90:1) circle (.30);
		\draw[black,fill=black] (210:1) circle (.35);
		\draw[black,fill=black] (330:1) circle (.35);
		\draw[black,line width=0.28cm ] (210:1) -- (330:1);
		\draw[opacity=0] (0:1.2) circle (0.1);
	}}

\newsavebox\vvvbox
\savebox\vvvbox{\tikz{
		\draw[black,fill=black] (90:1) circle (.35);
		\draw[black,fill=black] (210:1) circle (.35);
		\draw[black,fill=black] (330:1) circle (.35);
		\draw[opacity=0] (0:1.2) circle (0.1);
	}}
\newcommand{\vvv}{\mathord{\scaleobj{1.2}{\scalerel*{\usebox{\vvvbox}}{x}}}}
\newcommand{\pivvv}{\pi_{\vvv}}

\newsavebox\evbox
\savebox\evbox{\tikz{
		\draw[black,fill=black] (90:1) circle (.35);
		\draw[black,fill=black] (210:1) circle (.35);
		\draw[black,fill=black] (330:1) circle (.35);
		\draw[black,line width=0.28cm ] (210:1) -- (330:1);
		\draw[opacity=0] (0:1.2) circle (0.1);
	}}
\newcommand{\ev}{\mathord{\scaleobj{1.2}{\scalerel*{\usebox{\evbox}}{x}}}}

\newsavebox\eebox
\savebox\eebox{\tikz{
		\draw[black,fill=black] (90:1) circle (.35);
		\draw[black,fill=black] (210:1) circle (.35);
		\draw[black,fill=black] (330:1) circle (.35);
		\draw[black,line width=0.28cm ] (90:1) -- (330:1);
		\draw[black,line width=0.28cm ] (90:1) -- (210:1);
		\draw[opacity=0] (0:1.2) circle (0.1);
	}}
\newcommand{\ee}{\mathord{\scaleobj{1.2}{\scalerel*{\usebox{\eebox}}{x}}}}

\newsavebox\eeebox
\savebox\eeebox{\tikz{
		\draw[black,fill=black] (90:1) circle (.35);
		\draw[black,fill=black] (210:1) circle (.35);
		\draw[black,fill=black] (330:1) circle (.35);
		\draw[black,line width=0.28cm ] (90:1) -- (330:1);
		\draw[black,line width=0.28cm ] (90:1) -- (210:1);
		\draw[black,line width=0.28cm ] (210:1) -- (330:1);
		\draw[opacity=0] (0:1.2) circle (0.1);
	}}

\makeatletter
\newcommand{\overrighharpoonup}[1]{\ThisStyle{%
		\vbox {\m@th\ialign{##\crcr
				\rightharpoonupfill \crcr
				\noalign{\kern-\p@\nointerlineskip}
				$\hfil\SavedStyle#1\hfil$\crcr}}}}

\def\rightharpoonupfill{%
	$\SavedStyle\m@th\mkern+0.8mu\cleaders\hbox{$\shortbar\mkern-4mu$}\hfill\rightharpoonuptip\mkern+0.8mu$}

\def\rightharpoonuptip{%
	\raisebox{\z@}[2pt][1pt]{\scalebox{0.55}{$\SavedStyle\rightharpoonup$}}}

\def\shortbar{%
	\smash{\scalebox{0.55}{$\SavedStyle\relbar$}}}
\makeatother

\makeatletter
\newcommand{\overlefharpoonup}[1]{\ThisStyle{%
		\vbox {\m@th\ialign{##\crcr
				\leftharpoonupfill \crcr
				\noalign{\kern-\p@\nointerlineskip}
				$\hfil\SavedStyle#1\hfil$\crcr}}}}

\def\leftharpoonupfill{%
	$\SavedStyle\m@th\mkern+0.8mu\cleaders\hbox{$\shortbar\mkern-4mu$}\hfill\leftharpoonuptip\mkern+0.8mu$}

\def\leftharpoonuptip{%
	\raisebox{\z@}[2pt][1pt]{\scalebox{0.55}{$\SavedStyle\leftharpoonup$}}}

\makeatletter
\newsavebox\myboxA
\newsavebox\myboxB
\newlength\mylenA

\newcommand*\xoverline[2][0.75]{%
	\sbox{\myboxA}{$\m@th#2$}%
	\setbox\myboxB\null
	\ht\myboxB=\ht\myboxA%
	\dp\myboxB=\dp\myboxA%
	\wd\myboxB=#1\wd\myboxA
	\sbox\myboxB{$\m@th\overline{\copy\myboxB}$}
	\setlength\mylenA{\the\wd\myboxA}
	\addtolength\mylenA{-\the\wd\myboxB}%
	\ifdim\wd\myboxB<\wd\myboxA%
	\rlap{\hskip 0.5\mylenA\usebox\myboxB}{\usebox\myboxA}%
	\else
	\hskip -0.5\mylenA\rlap{\usebox\myboxA}{\hskip 0.5\mylenA\usebox\myboxB}%
	\fi}
\makeatother

\DeclareSymbolFont{symbolsC}{U}{txsyc}{m}{n}
\SetSymbolFont{symbolsC}{bold}{U}{txsyc}{bx}{n}
\DeclareFontSubstitution{U}{txsyc}{m}{n}
\DeclareMathSymbol{\strictif}{\mathrel}{symbolsC}{74}

\title{Lagrangians are attained as uniform Tur\'an densities}

\author[D.~King]{Dylan King}
\address{Mathematics Department, California Institute of Technology, Pasadena, USA}
\email{dking@caltech.edu}

\author[M.~Sales]{Marcelo Sales}
\address{Mathematics Department, University of California Irvine, Irvine, USA}
\email{mtsales@uci.edu}
\thanks{The second author is supported by US Air force grant FA9550-23-1-0298.}

\author[B.~Sch\"ulke]{Bjarne Sch\"ulke}
\address{Extremal Combinatorics and Probability Group, Institute for Basic Science, Daejeon, South Korea}
\email{schuelke@ibs.re.kr}
\thanks{The third author is supported by the Young Scientist Fellowship IBS-R029-Y7.}

\thanks{The work on this article was supported by the SQuaRE ``Variants of the hypergraph Tur\'an problem'' at the American Institute of Mathematics.}

\begin{document}
\maketitle
\begin{abstract}

The study of uniform Tur\'an densities was initiated in the 1980s by Erd\H{o}s and S\'os.
Given a~$3$-graph~$F$, the uniform Tur\'an density of~$F$,~$\pi_{\vvv}(F)$, is defined as the infimum~$d\in[0,1]$ such that every~$3$-graph~$H$ in which every linearly sized~$S\subseteq V(H)$ induces at least~$(d+o(1))\binom{\vert S\vert}{3}$ edges must contain a copy of~$F$.
Disproving Erd\H{o}s's famous jumping conjecture, Frankl and R\"odl showed that the set of Tur\'an densities is not well-ordered.
We prove an analogous result for the uniform Tur\'an density, namely that the set~$\Pi^{(3)}_{\vvv,\infty}=\{\pi_{\vvv}(\cF) : \cF\text{ a family of }3\text{-graphs} \}$ is not well-ordered.
This is a consequence of a more general result, which in particular implies that for every Lagrangian~$\Lambda$ of a~$3$-graph and integer~$1 \leq t \leq 6$ we have~$\frac{t}{6}\Lambda\in \Pi^{(3)}_{\vvv,\infty}$.

\end{abstract}
\section{Introduction}\label{sec:introduction}

Tur\'an's paper~\cite{T:41} in which he shows that complete, balanced,~$(r-1)$-partite graphs maximise the number of edges among all~$K_r$-free graphs is often seen as the starting point of extremal combinatorics.
	Already in this paper, he asked for similar results for hypergraphs.
	More than $80$ years later, still little is known in this regard.
	
	A~$k$-uniform hypergraph (or~$k$-graph)~$H=(V,E)$ consists of a vertex set~$V$ and an edge set~$E\subseteq V^{(k)}=\{e\subseteq V:\vert e\vert=k\}$.
	The extremal number for~$n\in\mathds{N}$ and a~$k$-graph~$F$ is the maximum number of edges in an~$F$-free~$k$-graph on~$n$ vertices and it is denoted by~$\ex(n,F)$.
	The Tur\'an density of~$F$ is~$\pi(F)=\lim_{n\to\infty}\frac{\ex(n,F)}{\binom{n}{k}}$ (this limit was shown to exist by monotonicity in~\cite{KNS:64}).
	The aforementioned result by Tur\'an was later generalised to the Erd\H{o}s-Stone-Simonovits theorem~\cites{ES:46,ES:66}, which determines the Tur\'an density of any graph~$F$ to be~$\frac{\chi(F)-2}{\chi(F)-1}$ (where~$\chi(F)$ is the chromatic number of~$F$).
	
	For hypergraphs, only few (non-trivial) Tur\'an densities are known exactly, one prominent example being the Fano plane~\cites{DF:00,FS:05,KS:05,BR:19}.
	The Tur\'an density of the complete~$3$-graph on four vertices,~$K_4^{(3)}$, is still not known, although it is widely conjectured to be~$5/9$.
	In fact, the problem is even open for the~$3$-graph that consists of three edges on four vertices, denoted by~$K_4^{(3)-}$.
	For more on the hypergraph Tur\'an problem, we refer to the survey by Keevash~\cite{K:11}.

Given the difficulty of determining the Tur\'an density of a specific hypergraph, one can also study the possible values of the function $\pi$.
Given a (possibly infinite) family~$\mathcal{F}$ of~$k$-graphs and~$n\in\mathds{N}$,~$\ex(n,\cF)$ is the maximum number of edges a~$k$-graph on~$n$ vertices can have without containing any member of~$\cF$ as a subhypergraph.
Then~$\pi(\mathcal{F})$ is defined as before.
Now the set  
\begin{align*}
\Pi_{\infty}^{(k)} = \{\pi(\mathcal{F}) : \mathcal{F} \text{ is a family of } k\text{-graphs}\}
\end{align*}  
is the set of possible Tur\'{a}n densities (of families) of~$k$-graphs.
Note that for $k=2$, the aforementioned result by Erd\H{o}s, Stone, and Simonovits implies that~$\Pi_{\infty}^{(2)} = \left\{ \frac{j-1}{j} : j \in\mathds{N} \right\}$.

We call~$\alpha\in[0,1]$ a \emph{jump} in~$\Pi_{\infty}^{(k)}$ if there exists some~$\varepsilon>0$ such that~$\Pi_{\infty}^{(k)}\cap(\alpha,\alpha+\varepsilon)=\emptyset$, otherwise we call~$\alpha$ a non-jump (in~$\Pi_{\infty}^{(k)}$).
For $k\geq 3$, Erd\H{o}s~\cite{E:64} initiated the investigation of~$\Pi_{\infty}^{(k)}$ by showing that~$\Pi^{(k)}_{\infty}\cap(0,\frac{k!}{k^k})=\emptyset$, i.e., that~$0$ is a jump in~$\Pi_{\infty}^{(k)}$.
Erd\H{o}s's famous jumping conjecture asserted that, analogous to the graph case, the set $\Pi_{\infty}^{(k)}$ is well-ordered.
This conjecture was disproved by Frankl and R\"{o}dl~\cite{FR:84}, who showed that for every integer~$k \geq 3$, there exists a non-jump.  
Subsequently, several non-jumps have been found (see, for example,~\cite{P:07, P:08, FPRT:07, PY:21}), as well as one further jump~\cite{BT:10}.

	In this work, we consider a variant of the hypergraph Tur\'an problem that was suggested by Erd\H{o}s and S\'os~\cite{E:90,ES:82} who asked for the maximum density of~$F$-free hypergraphs~$H$ if we require~$H$ to be uniformly dense.
	More precisely, given~$d\in[0,1]$ and~$\eta>0$, a~$3$-graph hypergraph~$H=(V,E)$ is called \emph{uniformly~$(d,\eta)$-dense} if for every~$U\subseteq V$, we have~$\vert U^{(3)}\cap E\vert\geq d\binom{\vert U\vert}{3}-\eta\vert V\vert^3$.
	The uniform Tur\'an density of a~$3$-graph~$F$ is defined as
	\begin{align}\label{eqn:nonstandard_defn}
		\pivvv(F)=\sup\{d\in[0,1]:&\text{ for every }\eta>0\text{ and }n\in\mathds{N},\text{ there is an}\\ &F\text{-free, uniformly }(d,\eta)\text{-dense }3\text{-graph }H\text{ with }\vert V(H)\vert\geq n\}\,\nonumber.
	\end{align}
	
	Erd\H{o}s and S\'os specifically asked to determine~$\pivvv(K_4^{(3)})$ and~$\pivvv(K_4^{(3)-})$.
	Similarly as with the original Tur\'an density, these problems turned out to be very difficult.
	Only recently, Glebov, Kr\'al', and Volec~\cite{GKV:16} and Reiher, R\"odl, and Schacht~\cite{RRS:18} independently solved the latter, showing that~$\pivvv(K_4^{(3)-})=1/4$ and thus confirming a conjecture by Erd\H{o}s and S\'os.
	We refer to Reiher's survey~\cite{R:20} for a full description of the landscape of extremal problems in uniformly dense hypergraphs.

Similarly to the original Tur\'{a}n density, one can also define the uniform Tur\'{a}n density~$\pi_{\vvv}(\cF)$ of a family~$\cF$ of~$3$-graphs. Let  
$$\Pi^{(3)}_{\vvv,\infty}=\{\pi_{\vvv}(\mathcal{F}):\mathcal{F}\text{ is a family of }3\text{-graphs}\}$$  
be the set of all possible uniform Tur\'an densities.
The first result in this paper is that~$\Pi^{(3)}_{\vvv,\infty}$ is not well-ordered.
This can be seen as an analogous result to the aforementioned theorem by Frankl and R\"{o}dl for~$\Pi^{(3)}_{\infty}$.

    \begin{theorem}\label{thm:non-jump}
        The set~$\Pi^{(3)}_{\vvv,\infty}$ is not well-ordered.
    \end{theorem}

    This theorem is a consequence of our main result, which implies that much of the structure of~$\Pi^{(3)}_\infty$ transfers to~$\Pi^{(3)}_{\vvv,\infty}$. To state our main result precisely, we recall some notation around Lagrangians, which go back to the work of Motzkin and Straus~\cite{MS:65}.
    Denote the standard $(r-1)$-simplex by $$\bbS_r := \{(x_1,\dots,x_r) \in [0,1]^r \ : \ x_1+\dots+x_r = 1 \}\,.$$
    Let~$F=(V,E)$ be a~$k$-graph on~$r$ vertices (we identify~$V=[r]$).
    Then the \emph{Lagrange polynomial} of $F$ is given by $$\lambda_F(x_1,\dots,x_r):=k!\sum_{i_1\dots i_k \in E}\prod_{\ell=1}^{k}x_{i_{\ell}}\,.$$    
    The \emph{Lagrangian} of~$F$ is defined as $\Lambda_F := \max_{\bx \in \bbS_r} \lambda_F(\bx)\,,$
    and we write the set of all such Lagrangians as $\Lambda^{(k)}:=\{\Lambda_F :\: F \text{ is a } k\text{-graph} \}\,$.

    Tur\'an densities and Lagrangians are intimately connected as described by Brown and Simonovits~\cite{BS:84} and Pikhurko~\cite{P:12}:
    \begin{align}\label{eq:brown_simonovits}
        \Lambda^{(k)} \subseteq  \Pi^{(k)}_{\infty}  = \overline{\Lambda^{(k)}}\,.
    \end{align}

    The main result of this work is that the set of scaled Lagrangians~$\frac{1}{6}\Lambda^{(3)}=\{\frac{1}{6}\Lambda_F:F\text{ is a }3\text{-graph}\}$ is contained in~$\Pi^{(3)}_{\vvv,\infty}$ (in fact, any integer multiple of~$\frac{1}{6}$ at most~$1$ can be taken).
    \begin{theorem}\label{thm:main_short}
    For any integer~$1 \leq t \leq 6$ we have~$\frac{t}{6}\Lambda^{(3)} \subset \Pi_{\vvv,\infty}^{(3)}\,$
    \end{theorem}

    This theorem combined with~\eqref{eq:brown_simonovits} implies that some known topological properties of~$\Pi_{\infty}^{(3)}$ carry over to~$\Pi_{\vvv,\infty}^{(3)}$. In particular, the set $\Pi^{(3)}_{\vvv,\infty}$ contains irrational values (see Section~\ref{sec:remarks}). Theorem~\ref{thm:main_short} will follow from a more general result (Theorem~\ref{thm:lambda_star}), which states that any ``Lagrangian of a palette'' is attained as the uniform Tur\'an density of some family. We remark that we also obtained similar results for other variants of the Tur\'{a}n density, namely~$\pi_{\ee}$ and $\pi_{\ev}$. 

    The remainder of this work is organised as follows. In Section~\ref{sec:prelim}, we establish the required notation for the three variants of quasirandomness and introduce some tools we need for the proof.
    In Section~\ref{sec:palettes} we define palette constructions and extend the notion of hypergraph Lagrangians to them, so that we are able to state Theorem~\ref{thm:lambda_star}. 
    In Section~\ref{sec:proof_of_intro}, we show how Theorem~\ref{thm:lambda_star} implies Theorem~\ref{thm:main_short} which in turn gives Theorem~\ref{thm:non-jump}.
    The proof of Theorem~\ref{thm:lambda_star} is the content of Section~\ref{sec:proof}.
    In Section~\ref{sec:remarks} we provide some concluding remarks.

\section{Preliminaries}\label{sec:prelim}

For a~$k$-graph~$H=(V,E)$ and~$V_1,\dots,V_k\subseteq V$, we write $E(V_1,\dots,V_k)=\{v_1\dots v_k:v_i\in V_i\text{ for all }i\in[k]\}$ and~$e(V_1,\dots,V_k)=\vert E(V_1,\dots,V_k)\vert$.
If~$H$ is a~$3$-graph and~$X,Y,Z\subseteq V$, we set
$$E_{\vvv}(X,Y,Z):= \{(x,y,z) \in X \times Y \times Z \ : \ xyz \in E(H)  \}\,.$$
If further~$P,Q\subseteq V^{(2)}$, we define
$$E_{\ev}(X,P):= \{(x,y,z) \ : \ x \in X, yz\in P, \ xyz \in E(H)  \}\,.$$
$$ K_{\ee}(P,Q):= \{(x,y,z) \ : \ xy \in P, yz\in Q \}\,.$$
$$E_{\ee}(P,Q):= \{(x,y,z)\in K_{\ee}(P,Q) \ : \ xyz \in E(H)  \}\,.$$
We denote the cardinalities of~$E_{\vvv}$,$E_{\ev}$, and~$E_{\ee}$ by~$e_{\vvv}$,$e_{\ev}$, and~$e_{\ee}$, respectively.

Using these we define what it means for a graph to be $(d,\eta,\star)$-dense.
\begin{dfn}\label{def:star_dense}
    A $3$-graph $H=(V,E)$ is $(d,\eta,\vvv)$-dense if
    $$e_{\vvv}(X,Y,Z) \geq d|X||Y||Z|-\eta |V|^3$$ holds for all~$X,Y,Z\subseteq V$.
    It is $(d, \eta, \ev)$-dense if $$e_{\ev}(X,P) \geq d|X||P|-\eta |V|^3$$
    holds for all~$X\subseteq V$ and~$P\subseteq V^{(2)}$.
    Lastly, it is $(d, \eta, \ee)$-dense if $$e_{\ee}(P,Q) \geq d|K_{\ee}(P,Q)|-\eta |V|^3$$ holds for all~$P,Q\subseteq V^{(2)}$.
\end{dfn}
Using these notions of density we can define the uniform Tur\'an density and two further quasirandom variants.
\begin{dfn}
For~$\star\in\{\vvv,\ev,\ee\}$ and a family of~$3$-graphs~$\cF$ we define
\begin{align*}
    \pi_\star(\cF):=\sup \{d \in [0,1] : &\text{ for all } \eta > 0 \text{ and } n \in \NN, \text{ there is an $\cF$-free}\\
    &\text{and $(d,\eta,\star)$-dense $3$-graph $H$ on at least $n$ vertices} \}\,.
\end{align*}
If~$\cF$ consists of a single~$3$-graph~$F$, we write~$\pi_{\star}(F)$ instead of~$\pi_{\star}(\cF)$.
Further, we set 
$$\Pi^{(3)}_{\star,\infty} := \{\pi_\star(\cF) \ : \ \cF \text{ is a family of }3\text{-graphs} \}\,.$$
\end{dfn}
\begin{remark}
At first glance we have just defined~$\pi_{\vvv}(\cF)$ in a way different than in Equation~\ref{eqn:nonstandard_defn} in Section~\ref{sec:introduction}, but the two notions are equivalent; see~\cite{R:20} for more details. For the remainder we will apply only the definition of~$(d,\eta,\star)$-density given here. 
\end{remark}

We now discuss the methods of graph regularity, which is the main tool in the proof of Theorem~\ref{thm:lambda_star}.
Given~$d\in[0,1]$ and~$\varepsilon>0$, a graph~$G=(V,E)$, and disjoint sets~$A,B\subseteq V$ we say that~$G(A,B)$ (or simply~$(A,B)$ if~$G$ is clear from the context) is~$(\varepsilon,d)$-regular if for all~$X\subseteq A$ and~$Y\subseteq B$ with~$\vert X\vert\geq\varepsilon\vert A\vert$ and~$\vert Y\vert\geq\varepsilon\vert B\vert$, we have~$\big\vert e(X,Y)-d\vert X\vert\vert Y\vert\big\vert\leq\varepsilon\vert X\vert\vert Y\vert$.
If~$(A,B)$ is~$(\varepsilon,d)$-regular for some~$d\in[0,1]$, we say that it is~$\varepsilon$-regular.
\begin{dfn}
    Let~$G=(V,E)$ be a graph, let~$V_1,\dots,V_t\subseteq V$ be disjoint, and let $\phi:E(G)\to[r]$ be an~$r$-colouring.
    For~$a\in[r]$ and~$ij\in[t]^{(2)}$, we define
    \begin{align*}
        G_a(V_i,V_j):&=\{xy : x\in V_i,\ y \in V_j,\ \phi(x,y) = a \}\,\text{ and }
    \end{align*}
    for~$a,b,c\in[r]$ and~$ijk\in[t]^{(3)}$, we define
    \begin{align*}
        G_{abc}(V_i,V_j,V_k):&= \{xyz : xy \in {G_a(V_i,V_j)},\ yz \in G_b(V_j,V_k),\ xz \in G_c(V_i,V_k) \}.
    \end{align*}
    If~$G$ is complete, we write~$G_a(V_i,V_j)=E_a(V_i,V_j)$ and~$G_{abc}(V_i,V_j,V_k)=T_{abc}(V_i,V_j,V_k)$.
    We call a pair of vertex sets $(V_i,V_j)$ $\epsilon$-colour-regular if $G_a(V_i,V_j)$ is $\epsilon$-regular for all~$a \in [r]$.
\end{dfn}
We now state a variant of the Multicolour Regularity Lemma.
It can be proved in the usual way, just now starting with the intervals as initial partition and subsequently refining it (see~\cite{RRS:18}).
\begin{theorem}[Multicolour Regularity Lemma]\label{thm:regularity}
    For $\epsilon > 0$ and $m,r \in \NN$, there exists $M=M(m,r,\epsilon)$ so that, for every~$n\geq M$, every $r$-coloured graph $G$ with vertex set $[n]$ admits a partition $[n]=V_0 \dcup V_1 \dcup \dots \dcup V_t $ with $m \leq t \leq M$ and
 \begin{enumerate}[label=(\arabic*),ref=\arabic*]
    \item\label{it:regV0} $|V_0| \leq \epsilon n$,
     \item\label{it:regVsize} $|V_1|= \dots = |V_t|$,
     \item\label{it:regreg} the pair~$(V_i,V_j)$ is~$\eps$-colour-regular for all but at most~$\epsilon \binom{t}{2}$ of the~$ij\in[t]^{(2)}$,
     \item\label{it:regint} and for each $i\in[t]$ there is some~$s_i \in [m]$ with~$V_i\subseteq I_{s_i}$, where~$[n]=I_1\dcup \dots \dcup I_m$ is a balanced partition of $[n]$ into $m$ consecutive intervals.
 \end{enumerate}
\end{theorem}

Lastly, let us mention the Counting Lemma.
\begin{lemma}[Counting Lemma]\label{lem:counting}
     Let~$H$ be a graph with vertex set~$\{v_1,\dots,v_h\}$.
     Then for every~$\delta>0$ there is an $\epsilon>0$ so that if~$G$ is a graph with vertex set~$V_1\dcup\dots\dcup V_h$ and~$G(V_i, V_j)$ is~$(\epsilon,d^{ij})$-regular for all~$ij\in[h]^{(2)}$, then the number of copies of $H$ in~$G$ with $v_i \in V_i$ for each $i \in [h]$ is $$(1 \pm \delta)\prod_{v_iv_j \in E(H)}d^{ij}\prod_{i=1}^h |V_i|\,.$$
\end{lemma}

    \section{Palette Lagrangians}\label{sec:palettes}
    Historically, the known lower bound constructions for~$\pi_{\vvv}(F)$ were given by hypergraphs following the colour patterns in ``palettes''.
    To make this more precise, we introduce the following notation.

    \begin{dfn}\label{dfn:palLag}
        A \emph{palette} is a finite set~$\ccP$ of ordered~$3$-subsets of~$\NN$.
        The \emph{set of colours} of~$\ccP$ is~$\Phi(\ccP)=\bigcup_{p\in\ccP}\{a:a\in p\}$.
        A  \emph{weighting} of a palette $\ccP$ is a vector~$\bx=(x_a)_{a \in \Phi(\ccP)}\in\bbS_{\vert\Phi(\ccP)\vert}$.
        Given a palette~$\ccP$ with a weighting~$\bx$, set $$\lambda^{\vvv}_\ccP(\bx):=\sum_{(a,b,c) \in \ccP}x_ax_bx_c\,.$$
\end{dfn}
\begin{dfn}
    For a palette~$\ccP$ and~$a,b\in\Phi(\ccP)$ (not necessarily distinct), the ordered degree and codegree Lagrangians at~$a$ and at~$a,b$, respectively, are defined as
    \begin{align*}
        \lambda^{a}_{\ccP}(\bx)&:=\min\Big\{ \sum_{(a,c,d) \in \ccP}x_cx_d, \sum_{(c,a,d) \in \ccP}x_cx_d, \sum_{(c,d,a) \in \ccP}x_cx_d \Big\}\, \text{ and }\\
        \lambda^{a,b}_{\ccP}(\bx)&:=\min\Big\{ \sum_{(a,b,c) \in \ccP}x_c,\sum_{(b,a,c) \in \ccP}x_c, \sum_{(a,c,b) \in \ccP}x_c,\sum_{(b,c,a) \in \ccP}x_c, \sum_{(c,a,b) \in \ccP}x_c, \sum_{(c,b,a) \in \ccP}x_c \Big\}\,.\\
    \end{align*}
    We denote the minimum over those~$a\in\Phi(\ccP)$ or~$a,b\in\Phi(\ccP)$ with~$x_a>0$ or~$x_a,x_b>0$, respectively, as
    \begin{align*}
        \lambda^{\ev}_\ccP(\bx) &:= \min_{\substack{a \in \Phi(\ccP):\\ x_a >0}}\lambda^a_\ccP(\bx)\, \text{ and }\\
        \lambda^{\ee}_\ccP(\bx) &:= \min_{\substack{a,b \in \Phi(\ccP):\\ x_ax_b >0}}\lambda^{a,b}_\ccP(\bx).
    \end{align*}
    Finally, for~$\star \in \{\vvv,\ev,\ee\}$, maximising over~$\bx\in\bbS_{\vert\Phi(\ccP)\vert}$
    \begin{equation*}
        \Lambda^\star_\ccP:= \max_{\bx \in S_{|\Phi(\ccP)|}}\lambda^\star_\ccP(\bx)
    \end{equation*}
    defines~$\Lambda^\star_\ccP$, the palette Lagrangian of~$\ccP$, which we also call the~$\star$-Lagrangian of~$\ccP$ when we wish to emphasise~$\star$.
\end{dfn}

    The set of values obtained as the~$\star$-Lagrangian of a palette is denoted as~$$\Lambda_{\star}^{\textup{pal}}:= \{\Lambda_\ccP^{\star} \ : \ \ccP \text{ is a palette} \}.$$

As mentioned above, palettes can be used to obtain lower bounds for~$\pi_{\star}$.

\begin{dfn}
    A $3$-graph $F=(V,E)$ on $n$ vertices \emph{satisfies} a palette $\ccP$ if there exists a vertex ordering $v_1,\dots,v_n$ and colouring $\phi: \binom{V}{2} \to \Phi(\ccP)$ so that for each $v_iv_jv_k \in E$ with $i<j<k$ we have $(\phi(v_iv_j),\phi(v_jv_k),\phi(v_iv_k)) \in \ccP$.
    We say that a family of $3$-graphs $\cF$ satisfies $\ccP$ if some $F \in \cF$ satisfies $\ccP$.
\end{dfn}

Given a palette~$\ccP$ and~$\eta>0$ one can use probabilistic methods to obtain large hypergraphs~$H$ which satisfy~$\ccP$ and are~$(\Lambda_{\ccP}^{\star},\eta,\star)$-dense, giving the following folklore result. 

\begin{lemma}\label{lem:palette_lb}
    If a family of~$3$-graphs~$\cF$ does not satisfy a palette~$\ccP$ then~$\pi_{\star}(\cF) \geq \Lambda^\star_\ccP$ for each~$\star \in  \{\vvv,\ev,\ee \}$.
\end{lemma}

As pointed out in (\ref{eq:brown_simonovits}), the set of Lagrangian values $\Lambda^{(3)}$ is intimately connected to the set of Tur\'{a}n densities $\Pi_{\infty}^{(3)}$.
It would be interesting to extend such a connection to uniform Tur\'{a}n densities.
Very recently, Lamaison \cite{L:24} showed that every uniform Tur\'{a}n density can be approximated by the Lagrangian of a palette, so in particular, $\Pi_{\vvv,\infty}^{(3)}\subseteq \overline{\Lambda_{\vvv}^{\textup{pal}}}$.
Our general result (which will imply Theorems~\ref{thm:non-jump} and~\ref{thm:main_short}) is that the~$\star$-Lagrangian of any palette~$\ccP$ can be obtained as~$\pi_\star(\cF)$ for some family~$\cF$.

    \begin{theorem}\label{thm:lambda_star}
    $\Lambda_\star^{\textup{pal}} \subseteq \Pi^{(3)}_{\star,\infty}$ for $\star \in \{\vvv,\ev,\ee \}\,.$
    \end{theorem}

    An immediate consequence of Theorem \ref{thm:lambda_star} is that the inclusions in (\ref{eq:brown_simonovits}) hold in one direction for $\pi_{\vvv}$, i.e., we have
    \begin{align*}
        \Lambda^{\textup{pal}}_{\vvv} \subseteq  \Pi^{(3)}_{\vvv,\infty}  \subseteq \overline{\Lambda^{\textup{pal}}_{\vvv}}\,.
    \end{align*}
    Let us remark that the main difficulty of our results lies in the fact that while palette Lagrangians (and all constructions) are defined with respect to a fixed distribution of colours on the pairs of vertices, a given uniformly~$(d,\eta,\vvv)$-dense $3$-graph that satisfies a palette~$\ccP$ does not need to follow such a fixed distribution.
    Addressing this issue is the main contribution of this paper, and it is captured in the following theorem, which states that the uniform density of any $3$-graph which satisfies a palette~$\ccP$ cannot substantially exceed~$\Lambda_{\ccP}^{\vvv}$.

    \begin{theorem}\label{thm:palette_ub}
        For every palette~$\ccP$ and~$\nu>0$ there are $C>0$ and~$N \in \NN$ such that the following holds.
        If $H$ is a $3$-graph on $n>N$ vertices which satisfies $\ccP$ and $H$ is uniformly $(d,\eta,\vvv)$-dense for some~$d\in[0,1]$ and~$\eta>0$, then $d \leq \Lambda^{\vvv}_\ccP+C\eta+\nu$.
    \end{theorem}

\section{Proofs of Theorems \ref{thm:main_short} and \ref{thm:non-jump}}\label{sec:proof_of_intro}

\begin{proof}[Proof of Theorem~\ref{thm:main_short} using Theorem~\ref{thm:lambda_star}]
Suppose that~$\Lambda \in \Lambda^{(3)}$ is the Lagrangian of a~$3$-graph $F=(V,E)$ and~$1 \leq t \leq 6$ is given.
We will show that~$\frac{t}{6}\Lambda$ is contained in~$\Lambda_{\vvv}^{\text{pal}}$, in which case we will be done by appealing to Theorem~\ref{thm:lambda_star}.
Without loss of generality, we may write~$V=[r]$.
Given~$a,b,c \in [r]$,~write $p_t(a,b,c)$ as the set whose elements are the first~$t$ permutations of $(a,b,c)$, when listed lexicographically.
We set~$\ccP$ to be the palette~$$\ccP = \bigcup_{abc \in E, a<b<c}p_t(a,b,c)\,.$$
That is, the colours of $\ccP$ are the vertices of $F$ and each edge of $F$ becomes~$t$ triplets in~$\ccP$, where each of the~$t$ appearances differs only in their ordering.
Then~$\lambda^{\vvv}_\ccP(x_1,\dots,x_r) = \frac{t}{6}\lambda_F(x_1,\dots,x_r)$ so that the maxima~$\frac{1}{6}\Lambda = \Lambda_{\ccP}^{\vvv}$ agree as well and we are done.

\end{proof}

\begin{proof}[Proof of Theorem~\ref{thm:non-jump} using Theorem~\ref{thm:main_short}]
The existence of a non-jump due to Frankl and R\"odl \cite{FR:84} guarantees the existence of~$\alpha \in (2/9,1)$ and a strictly decreasing sequence~$\{\pi_n\}_{n\in\mathds{N}} \subseteq \Pi_\infty^{(3)}$ with~$\lim_{n \to \infty} \pi_n = \alpha.$
By~\eqref{eq:brown_simonovits} this implies the existence of~$\lambda \in (2/9,1)$ and a strictly decreasing sequence~$\lambda_n \in \Lambda^{(3)}$ with~$\lim_{n \to \infty} \lambda_n = \lambda$ (here every element~$\pi_n\in\Pi_{\infty}^{(3)}=\overline{\Lambda^{(3)}}$ is approximated sufficiently well by some element in~$\Lambda^{(3)}$).
Finally Theorem~\ref{thm:main_short} implies the existence of an infinite strictly decreasing sequence in $\Pi_{\vvv,\infty}^{(3)}$, whence this set is not well-ordered.
\end{proof}

\section{Proof of Theorem \ref{thm:lambda_star}}\label{sec:proof}

We are interested mainly in hypergraphs which satisfy a palette~$\ccP$ but can actually find success whenever we are close to satisfying~$\ccP$. Formally, given a~$3$-graph~$H=(V,E)$ and~$\alpha>0$, we say that~$H$ $\alpha$\emph{-almost satisfies} a palette~$\ccP$ if we can remove at most~$\alpha\vert V\vert^3$ edges from~$H$ such that the resulting~$3$-graph satisfies~$\ccP$.
The following is a more general version of Theorem~\ref{thm:palette_ub} and is the core element of this work.
\begin{theorem}\label{thm:palette_ubgen}
    Let $\nu >0$, $\star \in \{\vvv,\ev,\ee\}$, and $\ccP$ be a palette.
    Then there are $\alpha$, $C>0$ and $N \in \NN$ such that the following holds.
    If~$H$ is a~$3$-graph on~$n\geq N$ vertices that~$\alpha$-almost satisfies~$\ccP$ and~$H$ is~$(d,\eta,\star)$-dense for some~$d\in[0,1]$ and~$\eta>0$, then $d \leq \Lambda^{\star}_\ccP+C\eta+\nu$.
\end{theorem}
First we show how to apply this to obtain Theorem~\ref{thm:lambda_star}.
\begin{proof}[Proof of Theorem \ref{thm:lambda_star}]
    We have to show that every~$\star$-Lagrangian associated with a palette can be obtained as the uniform Tur\'an density of some family.
    To that end let $\ccP$ be a palette, $\Lambda^{\star}_\ccP$ the associated $\star$-Lagrangian, and define~$\cF$ as the (infinite) family of all those~$3$-graphs which do not satisfy~$\ccP$.
    Letting~$\nu>0$, Theorem~\ref{thm:palette_ubgen} yields~$\alpha$,~$C$,~$N$ and we claim that 
    \begin{equation*}
        \Lambda^{\star}_\ccP \leq \pi_{\star}(\cF) \leq \Lambda^{\star}_\ccP+3\nu\,,
    \end{equation*}
    in which case we are done since~$\nu>0$ was arbitrary.
    The first inequality is the result of Lemma \ref{lem:palette_lb}, so it only remains to check that $\pi_{\star}(\cF) \leq \Lambda^{\star}_\ccP+3\nu$.

   To this end, set~$d=\pi_{\star}(\cF)-\nu$ and~$\eta = \nu/C$, and note that by the definition of~$\pi_{\star}(\cF)$, there exists a~$(d,\eta,\star)$-dense $3$-graph~$H$ on~$n>N$ vertices which is $\cF$-free.
   In particular,~$H$ satisfies~$\ccP$.
    So $H$ satisfies the conditions required by Theorem~\ref{thm:palette_ubgen} (even with $\alpha =0$), which in turn discloses~$\pi_{\star}(\cF)-\nu=d \leq \Lambda^{\star}_\ccP+2\nu$ as needed.
\end{proof}

Now we turn to the proof of Theorem \ref{thm:palette_ubgen}, which will require a technical lemma to address the following concern.
Suppose for a moment that we are in the case that $\star = \ev$ and that~$H$ even satisfies $\ccP$ without having to remove any edges.
This equips us with an ordering of~$V(H)$ and a colouring of~$\binom{V(H)}{2}$.
Since~$H$ is~$(d,\eta,\ev)$-dense, to upper bound~$d$ we would want to find a large set of pairs~$P\subseteq \binom{V(H)}{2}$ and set of vertices~$X\subseteq V(H)$ with~$e_{\ev}(X,P)$ being small.
A natural choice is to select for $P$ all those pairs which receive colour $a$, where $a\in \Phi(\ccP)$ is chosen to witness the minimum with $\lambda^a_\ccP(\bx)=\Lambda^\star_\ccP$ when $\bx$ is optimal. Unfortunately it could be that this colour $a$ is vanishingly rare in $H$; in which case the resulting bound on $d$ will be weak since $P$ will be small. The following technical lemma will allow us to always choose a witness colour~$a$ which approximates $\Lambda^{\ev}_{\ccP}$ and is popular in the underlying colouring (and handles the analogous problem in the~$\ee$ case). Similar issues were addressed in Pikhurko's result~\cite{P:12}.

\begin{lemma}\label{lem:beta}
    Let $\ccP$ be a palette and~$\nu > 0$. Then there is a $\beta^{\ev}\in (0,1]$ so that
    \begin{equation*}
        \text{for all} \ \bx \in \bbS_{|\Phi(\ccP)|} \ \text{there exists} \ a \in \Phi(\ccP) \textup{ with } x_a \geq \beta^{\ev} \textup{ and }\lambda^a_\ccP(\bx) \leq \Lambda^{\ev}_\ccP+\nu .
    \end{equation*}
    Likewise there is a $\beta^{\ee}\in(0,1]$ so that
    \begin{equation*}
        \text{for all} \ \bx \in \bbS_{|\Phi(\ccP)|} \ \text{there exists} \ a,b \in \Phi(\ccP) \textup{ with } x_a,x_b \geq \beta^{\ee} \textup{ and }\lambda^{a,b}_\ccP(\bx) \leq \Lambda^{\ee}_\ccP+\nu.
    \end{equation*}
\end{lemma}
\begin{proof}
    Suppose first that~$\ccP$ is a palette and, for the sake of contradiction, that no~$\beta^{\ev}$ as desired exists.
    This yields a sequence~$\{\bx^n\}_{n\in\mathds{N}} \subseteq \bbS_{|\Phi(\ccP)|}$ so that for every~$a\in \Phi(\ccP)$ with~$\lambda^a_\ccP(\bx^n) \leq \Lambda^{\ev}_\ccP+\nu$, we have~$x^n_a <1/n$.
    Define this set of ``sparse'' colours as $A_n:= \{a \in \Phi(\ccP) \ : \ \lambda^a_\ccP(\bx^n) \leq \Lambda^{\ev}_\ccP+\nu \}$.
    Note that by definition of~$\Lambda^{\ev}_\ccP$,~$A_n$ is nonempty for each $n$ and that further~$A_n\neq\Phi(\ccP)$ for $n>|\Phi(\ccP)|$.
    Thus, by applying the pigeonhole principle over the $2^{|\Phi|}-2$ nonempty proper subsets of $\Phi(\ccP)$, we obtain a subsequence along which $A_n$ is constant, say~$A$, with $A \notin \{\emptyset,\Phi(\ccP)\}$.
    By the compactness of $\bbS_{|\Phi(\ccP)|}$ we find a further subsequence convergent to some $\bx^* \in \bbS_{|\Phi(\ccP)|}$.
    Note that by the definition of~$\bx^n$ and~$A$,~$x_a^*=0$ for every~$a\in A$.
    Further, the continuity of $\lambda^{a}_\ccP(\bx)$ in $\bx$ implies that for for $a \not \in A$ we have $\lambda^{a}_\ccP(\bx^*) \geq \Lambda^{\ev}_\ccP+\nu$.
    Then we find a contradiction~$\Lambda^{\ev}_\ccP \geq \lambda^{\ev}_\ccP(\bx^*) \geq \Lambda^{\ev}_\ccP + \nu$.

    The $\ee$ case is proved in a similar way; let $\ccP$ be a palette and suppose that a $\beta^{\ee}$ as desired does not exist.
    Then there is a sequence $\{\bx^n\}_{n\in\mathds{N}} \subseteq \bbS_{|\Phi(\ccP)|}$ so that all $a,b\in \Phi(\ccP)$ with $\lambda^{a,b}_\ccP(\bx^n) \leq \Lambda^{\ee}_\ccP+\nu$ satisfy $\min \{ x^n_a,x^n_b \} <1/n$.
    Let $$A_n:= \{a \in \Phi(\ccP) \ : \ x_a <1/n \text{ and }\lambda^{a,b}_\ccP(\bx^n) \leq \Lambda^{\ee}_\ccP+\nu\text{ for some }b \in \Phi \}\,.$$
    Just as before $A_n$ is nonempty for each $n$ and~$A_n\neq\Phi(\ccP)$ for $n>|\Phi(\ccP)|$.
    By appealing again to the pigeonhole principle, the compactness of~$\bbS_{|\Phi|}$, and the continuity of~$\lambda^{a,b}_\ccP(\bx)$, we obtain a set $A\subseteq \Phi(\ccP)$ with $A \notin\{\emptyset,\Phi(\ccP)\}$ and a point $\bx^* \in \bbS_{|\Phi|}$ so that
    \begin{itemize}
        \item for $a \in A$, $x^*_a = 0$,
        \item and for $a,b \not \in A$, $\lambda^{a,b}_\ccP(\bx^*) \geq \Lambda^{\ee}_\ccP+\nu$.
    \end{itemize}
    Then~$\Lambda^{\ee}_{\ccP} \geq \lambda^{\ee}_{\ccP}(\bx^*) \geq \Lambda^{\ee}_{\ccP}+\nu$ a contradiction.
\end{proof}

\begin{proof}[Proof of Theorem \ref{thm:palette_ubgen}]
    Let~$\ccP$ be a palette,~$\nu>0$, and denote~$\Phi:=\Phi(\ccP)$.
    First we establish some parameters.
    Depending on~$\star$ and~$\nu$, Lemma \ref{lem:beta} supplies~$\beta:=\beta^\star$; in the case~$\star=\vvv$ we set~$\beta=1$.
    We choose the remaining constants according to the following hierarchy.
    \begin{align}\label{eq:hierarchy}
        0<C^{-1}\ll N^{-1}\leq M^{-1}\ll m^{-1},\epsilon,\alpha\ll t_1^{-1}\ll\delta\ll\beta\ll\vert\Phi\vert^{-1},\nu\,.
    \end{align}
    In particular,~$\varepsilon$ is sufficiently small compared to~$\delta$ to apply Lemma \ref{lem:counting} on all graphs~$H$ on three vertices~$\{v_1,v_2,v_3\}$, and~$t_1\geq R(3;\ceil{1/\delta}^{|\Phi|})$, i.e., any colouring (of the edges) of~$K_{t_1}^{(2)}$ with $\ceil{1/\delta}^{|\Phi|}$ colours contains a monochromatic triangle.
    Further,~$\epsilon$ and~$\alpha$ are small enough, and~$m \in \NN$ large enough such that for all~$t \geq m$, any~$3$-graph on~$t$ vertices with at least $\binom{t}{3}-\epsilon t \binom{t}{2} - 8\sqrt{\alpha}t^3-\frac{2t^3}{m}$ edges contains a~$K_{t_1}^{(3)}$ (recall that~$\pi(K_{t_1}^{(3)})<1$).
    Moreover,~$M=M(m,|\Phi|,\epsilon)$ is chosen by applying Theorem~\ref{thm:regularity} to~$m$ and~$\varepsilon$, as well as~$\vert\Phi\vert$ here in place of~$r$ there, to yield an~$M=M(m,|\Phi|,\epsilon)$ such that the conclusion of Theorem~\ref{thm:regularity} holds.
    Further set~$C=8M^3/\beta^2$, and~$N = M$.

    Now let~$H$ be a~$3$-graph on~$n\geq N$ vertices that~$\alpha$-almost satisfies~$\ccP$ and is~$(d,\eta,\star)$-dense for some~$d\in[0,1]$ and~$\eta>0$.
    Let~$H_0\subseteq H$ be a~$3$-graph which satisfies $\ccP$, obtained from~$H$ by deleting at most $\alpha n^3$ edges. Since~$H$ is~$(d,\eta,\star)$-dense, to upper bound~$d$ we will find large sets of vertices or edges, depending on $\star \in \{\vvv,\ev,\ee\}$, with relatively few crossing~$3$-edges.
    
    Since~$H_0$ satisfies~$\ccP$, there is an ordering of~$V(H_0)=V(H)$ - w.l.o.g., we may assume that~$V(H_0)=[n]$ - and a colouring $\phi:\binom{[n]}{2} \to \Phi$ so that every $3$-edge in $H_0$ has as its ordered shadow an element of~$\ccP$.
    
    \begin{claim}
        There are three disjoint sets~$X_1,X_2,X_3\subseteq V$,~$d_a^{ij}\in[0,1]$, and~$x_a\in[0,1]$, for all~$ij\in[3]^{(2)}$ and~$a\in\Phi$ such that
        \begin{enumerate}[label=(\alph*),ref=\alph*]
            \item\label{it:cleanedsize} $n_1:=|X_1|=|X_2|=|X_3|\geq \frac{n}{2M}$,
            \item\label{it:cleanedorder} we have~$x<y$ for all~$x\in X_i$ and~$y\in X_j$ for all~$ij\in[3]^{(2)}$,
            \item\label{it:cleanedreg} the pair $(X_i,X_j)$ is $(\epsilon,d_a^{ij})$-regular for all~$ij\in[3]^{(2)}$ and $a \in \Phi$,
            \item\label{it:cleanedden} $d_a^{ij} \in [x_a,x_a+\delta)$ for all~$ij\in[3]^{(2)}$ and $a \in \Phi$,
            \item\label{it:cleanedsum} $\sum_{a\in \Phi} x_a \leq 1$, and
            \item\label{it:cleanededges} $e_{H \setminus H_0}(X_1,X_2,X_3)\leq\sqrt{\alpha}n_1^3$.
        \end{enumerate}
    \end{claim}
    \begin{proof}
    The definitions of~$M$ and~$N$ ensure that the conclusion of Theorem~\ref{thm:regularity} applies to~$H_0$ with the colouring~$\phi$ which yields a nonnegative integer~$t$ with $m \leq t \leq M$ and a partition~$[n]=V_0\dcup\dots\dcup V_t$ with properties~\eqref{it:regV0}-\eqref{it:regint}.
    Set~$n_1=\vert V_1\vert$.
    Since we will choose the sets~$X_i$ among the sets~$V_j$, with~$j\in[t]$, properties~\eqref{it:regV0} and~\eqref{it:regVsize} imply~\eqref{it:cleanedsize}.
    
    Next, consider the following auxiliary~$3$-graph $R$ with vertex set~$[t]$, where for~$i,j,k\in[t]$ the set~$ijk$ forms an edge if
    \begin{enumerate}[label=(\roman*),ref=\roman*]
        \item\label{it:R_reg} all three pairs $(V_i,V_j)$, $(V_j,V_k)$, and $(V_i,V_k)$ are $\epsilon$-colour-regular,
        \item\label{it:R_cleaned} $e_{H \setminus H_0}(V_i,V_j,V_k)\leq\sqrt{\alpha}n_1^3$,
        \item\label{it:R_order} and~$V_i\subseteq I_{s_i}$,~$V_j\subseteq I_{s_j}$, and $V_k\subseteq I_{s_k}$ for distinct~$s_i,s_j,s_k \in [m]$.
    \end{enumerate}
    Our goal is to show that $R$ is nearly complete and therefore will contain a large clique which in turn will admit an application of Ramsey's theorem with an appropriate colouring that will result in~$X_1$,~$X_2$, and~$X_3$ as desired.
    
    Due to property~\eqref{it:regreg}, condition~\eqref{it:R_reg} eliminates at most~$t \epsilon \binom{t}{2}$ potential $3$-edges from~$R$.
    To see how many triples fail condition~\eqref{it:R_cleaned}, set
    \begin{equation*}
        s = \Big\vert\Big\{ijk \in [t]^{(3)} \ : \ e_{H \setminus H_0}(V_i,V_j,V_k)>\sqrt{\alpha}n_1^3\Big\}\Big\vert\,.
    \end{equation*}
    Since~$H\setminus H_0$ contains at most~$\alpha n^3$ edges, we infer
    \begin{equation*}
         s\sqrt{\alpha} n_1^3 \leq \alpha n^3\,,
    \end{equation*}
    whence
    \begin{equation*}
        s \leq \sqrt{\alpha} \Big(\frac{n}{n_1}\Big)^3 \leq 8\sqrt{\alpha}t^3\,,
    \end{equation*}
    so condition~\eqref{it:R_cleaned} eliminates at most $8\sqrt{\alpha}t^3$ potential $3$-edges from $R$.
    
    Finally, note that due to~\eqref{it:regint} (and~\eqref{it:regVsize}), for any~$i\in[t]$, there are at most~$\frac{n}{mn_1}$ indices~$j\in[t]$ such that neither~$V_i<V_j$ nor~$V_j<V_i$.
    Thus, condition~\eqref{it:R_order} eliminates at most
    \begin{equation*}
        t^2\frac{n}{mn_1}\leq\frac{2t^3}{m}
    \end{equation*}
    potential edges from $R$.
    Therefore, $R$ contains at least $\binom{t}{3}-\epsilon t \binom{t}{2} - 8\sqrt{\alpha}t^3-\frac{2t^3}{m}$ edges and by our parameter choice we obtain a clique of size~$t_1$ in~$R$, w.l.o.g. with vertex set~$[t_1]$.
    Since we will select~$X_1$,~$X_2$, and~$X_3$ from amongst the~$V_i$ for~$i \in [t_1]$, the clique in $R$ will ensure they satisfy~\eqref{it:cleanedreg},~\eqref{it:cleanedorder}, and~\eqref{it:cleanededges}.

    Each pair $(V_i,V_j)$ with~$ij\in[t_1]^{(2)}$ is $\epsilon$-colour-regular with densities $d_a^{ij}$, where~$a\in\Phi$, such that $1=\sum_{a\in \Phi}d_a^{ij}$.
    For each~$ij\in[t_1]^{(2)}$ and $a \in \Phi$ we choose $x^{ij}_a \in \{0,\delta,2\delta,\dots,\delta(\ceil{1/\delta}-1)\}$ so that
    \begin{equation*}
           x^{ij}_a \leq d^{ij}_a < x^{ij}_a + \delta.
    \end{equation*}
    The vectors~$(x^{ij}_a)_{a \in \Phi}$ comprise a $\ceil{1/\delta}^{|\Phi|}$-colouring of~$K^{(2)}_{t_1}$.
    Hence, our choice of~$t_1$ guarantees a monochromatic triangle under this colouring. Let~$(x_a)_{a \in \Phi}$ be the colour of this triangle and w.l.o.g. say it is on vertex set~$[3]$.
    Setting~$X_1=V_1$,~$X_2=V_2$, and~$X_3=V_3$, this implies that the~$x_a$ satisfy~\eqref{it:cleanedden}.
    Combining ~$1=\sum_{a\in \Phi}d_a^{ij}$ together with~$x_a\leq d^{ij}_a$ implies~\eqref{it:cleanedsum}.
    Therefore, the~$X_1$,~$X_2$,~$X_3$,~$d_a^{ij}$, and~$x_a$ are as desired.
    \end{proof}

    Now let us finish the proof separately for each choice of~$\star$.
    
    \smallskip	
    {\it \hskip1em Case 1: $\star = \vvv$.}	
    \smallskip
    
    Appealing to Definition \ref{def:star_dense} for~$H$ on~$X_1,X_2,X_3$ yields
    \begin{align*}
        d & \leq \frac{\eta n^3+e_{\vvv}(X_1,X_2,X_3)}{n_1^3}\nonumber\,.
    \end{align*}

    Using~\eqref{it:cleanedsize} and~\eqref{it:cleanededges} (and that~$X_1$,~$X_2$, and~$X_3$ are mutually disjoint), this gives
    \begin{align}\label{eq:initial}
        d\leq \eta 8M^3+\sqrt{\alpha}+e_{H_0}(X_1,X_2,X_3)/n_1^3\,.
    \end{align}

    We are concerned now (as well as in the subsequent cases) with only~$3$-edges spanning~$X_1,X_2,X_3$ and will write~$T_{abc}$ for~$T_{abc}(X_1,X_2,X_3)$, suppressing the argument.
    Since~$H_0$ satisfies~$\ccP$ and because of~\eqref{it:cleanedorder}, we know that~$e_{H_0}(X_1,X_2,X_3)\leq\sum_{(a,b,c) \in \ccP}|T_{abc}|$.

    Next, apply the triangle counting lemma (Lemma~\ref{lem:counting}), which is possible due to~\eqref{it:cleanedreg}.
    Recalling the bounds in~\eqref{it:cleanedden}, this entails
    \begin{align}\label{eq:prelim}
        \frac{1}{n_1^3}\sum_{(a,b,c) \in \ccP}|T_{abc}|&\leq\sum_{(a,b,c) \in \ccP} (1 + \delta)(x_a+\delta)(x_b+\delta)(x_c+\delta)\nonumber \\
        &\leq\lambda^{\vvv}_\ccP(\bx)+\frac{\nu}{2} \,,
    \end{align}
    where the second inequality follows from Definition~\ref{dfn:palLag} and~\eqref{it:cleanedsum} (and the choice of the constants).
    Combining~\eqref{eq:initial} and~\eqref{eq:prelim} and recalling the choice of the constants, we finally obtain~$d\leq\Lambda^{\vvv}_\ccP+C\eta+\nu$ as desired.

    \smallskip	
    {\it \hskip1em Case 2:~$\star = \ev$.}	
    \smallskip

    Recall that we have~$\beta:=\beta^{\ev}$ so that for~$\bx = (x_a)_{a \in \Phi}$ there is $a \in \Phi$ so that $x_a \geq \beta$ and $\lambda^a_\ccP(\bx) \leq \Lambda^{\ev}_\ccP+\nu$. We may, w.l.o.g., assume that the minimum in $\lambda^a_\ccP(\bx)$ is obtained when $a$ is in the $(X_1,X_2)$ position of the palette.
    Now analogous reasoning as in the first case, appealing to Definition~\ref{def:star_dense} with vertices~$X_3$ and pairs~$E_a^{12}$, gives
    \begin{align}
        d & \leq \frac{\eta n^3+e_{\ev}(X_3,E_a^{12})}{d_a^{12}n_1^3} \leq \eta 8M^3/\beta+\sqrt{\alpha}/\beta+\frac{1}{d_a^{12}n_1^3}\sum_{b,c\in \Phi \text{ s.t. }(a,b,c) \in \ccP}|T_{abc}|\,,\nonumber
    \end{align}
    and the Lagrangian~$\lambda^{a}_{\ccP}(\bx)$ bounds third term above as 
        \begin{align}
        \frac{1}{d_a^{12}n_1^3}\sum_{b,c\in \Phi \text{ s.t. }(a,b,c) \in \ccP}|T_{abc}| & \leq \lambda^{\ev}_{\ccP}(\bx)+ \frac{\nu}{2} \nonumber
    \end{align}
    in which case we are done.

    \smallskip	
    {\it \hskip1em Case 3: $\star = \ee$.}	
    \smallskip
    
     As in Case 2 we have~$\beta:=\beta^{\ev}$ so that for~$\bx = (x_a)_{a \in \Phi}$, there is a pair of colours~$a,b \in \Phi$ with~$x_a,x_b \geq \beta$ and $\lambda^{a,b}_\ccP(\bx) \leq \Lambda^{\ee}_\ccP+\nu$.
     We may assume that the minimum in~$\lambda^{a,b}_\ccP(\bx)$ is obtained when $a$ and $b$ are in the $(X_1,X_2)$ and $(X_2,X_3)$ positions of the palette, respectively.
     Appealing to Definition~\ref{def:star_dense} on the pairs~$E_a^{12}$ and~$E_b^{23}$ yields
    \begin{align}
        d & \leq \frac{\eta n^3+e_{\ee}(E_a^{12},E_b^{23})}{|K_{\ee}(E_a^{12},E_b^{23})|}\nonumber
    \end{align}
    and although we cannot determine the denominator exactly, by applying the counting Lemma \ref{lem:counting} we have~$$|K_{\ee}(E_a^{12},E_b^{23})| \geq (1 - \delta)(d_a^{12}d_b^{23})n_1^3\,.$$
    Then~$$d \leq \eta 16M^3/\beta^2+2\sqrt{\alpha}/\beta^2+\frac{1}{(1 - \delta)(d_a^{12}d_b^{23})n_1^3}\sum_{c\in \Phi \text{ s.t. }(a,b,c) \in \ccP}|T_{abc}|$$
    and the third term is bounded by the Lagrangian~$\lambda^{a,b}_{\ccP}(\bx)$ ~$$\frac{1}{(1 - \delta)(d_a^{12}d_b^{23})n_1^3}\sum_{c\in \Phi \text{ s.t. }(a,b,c) \in \ccP}|T_{abc}| \leq \frac{1}{1-\delta}\left(\lambda^{a,b}_\ccP(\bx)+\frac{\nu}{4}\right) \leq \Lambda^{\ee}_{\ccP}+\frac{\nu}{2}$$
    and we are done.

    \end{proof}

\section{Concluding remarks}\label{sec:remarks}

In this work we establish that the Lagrangian of any palette is attained as the uniform Tur\'an density of some family of~$3$-graphs.
Let us briefly mention some other consequences of our main result.
First note that our methods convert any non-jump in~$\Pi_{\infty}^{(3)}$ to a non-jump in~$\Pi_{\vvv,\infty}^{(3)}$ (see the proof of Theorem~\ref{thm:non-jump}).

A further application of Theorem \ref{thm:main_short} is that it allows us to generate new members of~$\Pi_{\vvv,\infty}^{(3)}$ by computing Lagrangians. For instance, one can calculate the following Lagrangians, each of which gives an element of~$\Pi_{\vvv,\infty}^{(3)}$.

\begin{obs}
    Let $C_\ell^{(3)}$ be the $3$-uniform tight cycle of length~$\ell$.
    Then the Lagrangians are 
    \begin{equation*}
    \frac{1}{6}\Lambda_{C_\ell^{(3)}}=
        \begin{cases}
                            1/27 \text{ for $\ell=3$} \\
                            1/16 \text{ for $\ell=4$}\\
                            1/25 \text{ for $\ell=5$}\\
                            1/27 \text{ for $\ell\geq 6$}.
        \end{cases}
    \end{equation*}
    Let $F_{3,2}$ be the hypergraph with vertices $\{1,2,3,4,5\}$ and edges $\{123,145,245,345\}$. Then $\frac{1}{6}\Lambda_{F_{3,2}} = \frac{5 \sqrt{5}+63}{1922}\approx 0.0385954.$
    Hence, all these numbers are contained in~$\Pi_{\vvv,\infty}^{(3)}$.
\end{obs}

Finally, for the original Tur\'{a}n density Pikhurko \cite{P:12} proved that for certain constructions coming from iterated blowups, there exist finite families of graphs whose extremal constructions are precisely given by those blowups. In particular, as a consequence of his result, one can obtain that $\Lambda^{(3)}\subseteq \Pi_{\fin}^{(3)}$, which strengthens the result in (\ref{eq:brown_simonovits}).
It would be interesting to prove an analogous version of the result in \cite{P:12} for uniform Tur\'{a}n densities.


We remark that an application of the strong hypergraph removal lemma due to Avart, R\"odl, and Schacht~\cite{ARS:07} together with Theorem \ref{thm:palette_ubgen} implies the weaker statement that every palette Lagrangian can be approximated by uniform Tur\'{a}n densities of finite families, i.e.,~$\Lambda^{\textup{pal}}_{\vvv}\subseteq \overline{\Pi_{\vvv,\fin}^{(3)}}$.
Thus, a similar proof as that of Theorem~\ref{thm:non-jump} gives that~$\Pi^{(3)}_{\vvv,\fin}$ is not well-ordered, and that any non-jump in~$\Pi_{\infty}^{(3)}$ can be converted into a non-jump in~$\Pi_{\vvv,\fin}^{(3)}$.

\end{document}